\documentclass{amsart}
 
\usepackage[margin=1in]{geometry}
\usepackage{amssymb}
\usepackage{ytableau}
 \usepackage{mathtools}
 \usepackage{graphicx}
 \usepackage{pifont}
 \newcommand{\ai}{a_i}
 \newcommand{\aj}{a_j}
 \newcommand{\bi}{b_i}
 \newcommand{\bj}{b_j}
 \newcommand{\bk}{b_k}
 \DeclareMathOperator{\stab}{Stab}

 \usepackage{multicol}
\usepackage{mathrsfs}
\usepackage{color}
\usepackage{multicol}
\usepackage{amsthm}
\usepackage{float}
\usepackage{verbatim} 
\usepackage[vcentermath]{youngtab}
\usepackage{tikz}
\usepackage{tikz-cd}
\usetikzlibrary{arrows}
\usetikzlibrary{shapes.arrows}   
\usetikzlibrary{positioning}
\usepackage{mathrsfs}
\usepackage{amssymb}
\usepackage{mathtools}
\usepackage{python}

\definecolor{1}{rgb}{1,0.2,0.3}
\definecolor{2}{rgb}{0.1,0.3,0.5}
\definecolor{3}{rgb}{1,1,0}
\definecolor{4}{rgb}{255,255,255}

\newcommand{\lm}{\lambda}

\newtheorem{theorem}{Theorem}[section]
\newtheorem{corollary}[theorem]{Corollary}
\newtheorem{lemma}[theorem]{Lemma}
\newtheorem{conjecture}[theorem]{Conjecture}
\newtheorem{proposition}[theorem]{Proposition}
\newtheorem{definition}[theorem]{Definition}

\newtheorem{manualtheoreminner}{Theorem}
\newenvironment{manualtheorem}[1]{%
  \renewcommand\themanualtheoreminner{#1}%
  \manualtheoreminner
}{\endmanualtheoreminner}

\newtheorem{manualconjectureinner}{Conjecture}
\newenvironment{manualconjecture}[1]{%
  \renewcommand\themanualconjectureinner{#1}%
 \manualconjectureinner
}{\endmanualconjectureinner}

\theoremstyle{remark}
\newtheorem{remark}[theorem]{Remark}

\usepackage{caption}
\usepackage{subcaption}

\begin{document}

\tikzset
{
  x=.23in,
  y=.23in,
}

\title{Isotopy graphs of Latin tableaux}

\author{R.Karpman}
\address[R. Karpman]{Department of Mathematics \\ Otterbein University}
\email{karpman1@otterbein.edu}

\author{\'E. Rold\'an}
\address[\'E. Rold\'an]{Zentrum Mathematik,
Technische Universit\"at M\"unchen}
\email{erika.roldan@ma.tum.del}

\thanks{The second author received funding from the European Union's Horizon 2020 research and innovation program under the Marie Sk\l odowska-Curie grant agreement No.~754462.}

\date{\today}
\maketitle

\graphicspath{{Figures/}}

\begin{abstract}
Latin tableaux are a generalization of Latin squares, which first appeared in the early 2000's in a paper of Chow, Fan, Goemans, and Vondr\'{a}k.  Here, we extend the notion of isotopy, a permutation group action, from Latin squares to Latin tableaux.  We define isotopy graphs for Latin tableaux, which encode the structure of orbits under the isotopy action, and investigate the relationship between the shape of a Latin tableau and the structure of its isotopy graph.  Our main result shows that for any positive integer $d$, there is a Latin tableau whose isotopy graph is a $d$-dimensional cube. We show that most isotopy graphs are triangle-free, and we give a characterization of all the Latin tableaux for which the isotopy graph contains a triangle. We also establish that each connected component of an isotopy graph is regular, and give a formula for the degree of each vertex in a connected component of an isotopy graph which depends on both the shape of the tableau and the filling.
\end{abstract}

\maketitle{}

\section{Introduction}

A \emph{Latin square} is a square array of boxes filled with entries in the set $\{1,2,\ldots,n\}$, such that no entry appears twice in the same row or twice in the same column. Latin squares have long been an object of fascination for mathematicians. The first published example of a Latin square is due to a Korean mathematician, Choi Seokjeong, in 1700 \cite{San14}. In this note, we explore the combinatorics of \emph{Latin tableaux}, a generalization of Latin squares introduced by Timothy Chow and Brian Taylor \cite{CFGV02}.

Let $\lm = (\lm_1,\lm_2,\ldots,\lm_n)$ be a non-increasing sequence of positive integers--that is, a \emph{partition}. A \emph{Young diagram} of shape $\lm$ is an array of boxes arranged in left-justified rows, whose $i^{th}$ row has length $\lm_i$.  A \emph{Latin tableaux} is an assignment of integer entries to the boxes of a Young diagram such that row $i$ contains entries $1,2,\ldots,\lm_i$ in some order, and no entry appears twice in any column.  

Attempts to count $n \times n$ Latin squares for small $n$ date back at least to Euler in second half of the $18^{th}$ century.  See \cite{MMM07} for a historical account. There are several published formulas for the number of $n \times n$ Latin squares, although none that can be easily computed for large $n$. McKay and Wanless, for example, give an elegant formula in terms of permanents of matrices with entries in $\{-1, 1\}$ \cite{MW05}.

Two Latin squares are \emph{isotopic} if one can be obtained from the other by permuting rows, permuting columns, and/or permuting entries. Hence we may ask how many distinct \emph{isotopy classes} of $n \times n$ Latin squares exist for each $n$. McKay, Meynert and Myrvold computed the number of isotopy classes of $n \times n$ Latin squares for $n$ up to $10$ \cite{MMM07}.  Rather than constructing representatives for each class, they devised an approach which only required them to generate representatives for classes containing Latin squares with nontrivial symmetries. More recently, Hulpke, Kaski, and \"{O}sterg\r{a}rd counted isotopy classes for Latin squares of order 11, also using a blend of constructive and non-constructive techniques \cite{HKO11}. In \cite{MW05}, McKay and Wanless prove that as $n$ increases, the proportion of Latin squares of order $n$ with nontrivial symmetries rapidly approaches zero.

A \emph{partial Latin rectangle} is a rectangular array of boxes and a filling of some (but perhaps not all) of the boxes with positive integer entries, such that no entry appears twice in any row or column. We may think of Latin tableaux as partial Latin rectangles whose non-empty  boxes form the shape of a Young diagram. Isotopies of partial Latin rectangles are defined analogously to isotopies of Latin squares.

In \cite{falcon2019enumerating}, Falc\'{o}n and Stones study the enumeration of some sets of partial Latin rectangles which include Latin squares and Latin tableaux of a specific size. For fixed natural numbers $r,s,n$, and $m$ the set $PLR(r,s,n;m)$ of partial Latin rectangles contains all possible fillings of a Young diagram with rectangular shape $r \times s$  that are filled with $m$ entries from $[n]$, which denotes the set of integers from $1$ to $n$.

Falc\'{o}n and Stones compute the order of $PLR(r,s,n;m)$ for all $r \leq s \leq n \leq 6$.  In addition, they find the number of isotopy classes in $PLR(r,s,n:m)$ for each of these values of $r,s,n$ and $m$ using methods from algebraic geometry, together with an analysis of conjugacy classes in the group of isotopies acting on partial Latin rectangles of a given size \cite{falcon2019enumerating}.

\begin{figure}[ht]
    \centering
    \[
    \young(312,123,231) \hspace{1in} \young(3421,4312,123,21)
    \]
    \label{fig:tableaux}
    \caption{Examples of Latin tableaux}
\end{figure}

In this paper, we investigate the isotopy equivalence relation on Latin tableaux.  For a partition $\lm$, the \emph{isotopy graph} $\mathscr{G}(\lm)$ has as its vertices all Latin tableaux of shape $\lm$. Two vertices (tableaux) are connected by an edge if one can be obtained from the other by permuting a pair of rows, columns, or entries. Connected components in the isotopy graph correspond to isotopy classes, and the structure of the isotopy graph encodes the action of a permutation group, the \emph{isotopy group}, on the set of Latin tableaux of shape $\lm$.  The isotopy graph of a Latin tableau T of shape $\lambda$ is the connected component of $\mathscr{G}(\lm)$ containing $T$. In this paper, we explore the structure of isotopy graphs.

To the best of our knowledge, isotopy graphs have not been studied for Latin squares.  However, as explained in Remark \ref{coset}, the size of the isotopy graph of a Latin tableau is determined by the tableau's \emph{autotopy group}--that is, by the group of isotopies which fix the tableau.  Autotopy groups of Latin squares have been studied by several authors.  See, for example, the discussion and references in \cite{Sto13}. 

As $n$ increases, isotopy graphs for $n \times n$ Latin squares quickly become unmanageable in their size and complexity.  For tableaux of other shapes, however, we may obtain graphs whose structure is not only tractable but elegant. The heart of this paper is Section \ref{section:cube}, which deals with isotopy graphs isomorphic to cubes. Theorem \ref{theorem:cube} states that for any positive integer $d$, there is a Latin tableau whose isotopy graph is isomorphic to a $d$-dimensional cube, while Theorem \ref{cube_criterion} gives a characterization of the Latin tableaux which have cubes as their isotopy graphs. 

Earlier sections of the paper are devoted to general properties of isotopy graphs. Theorem \ref{clique_number} states that the clique number of an isotopy graph is either $1,2$ or $4$.  We show that most isotopy graphs are triangle-free, and list the partitions $\lm$ for which the isotopy graph $\mathscr{G}(\lm)$ contains a triangle. Theorem \ref{degree} gives a formula for the degree of vertex in a connected component of an isotopy graph, which depends on both the shape $\lm$ and the filling.

We end this section with a discussion of the history of Latin tableaux, and of our motivation for exploring them. Latin tableaux were introduced by Timothy Chow and Brian Taylor, who hoped to use them as a tool to prove Rota's basis conjecture, stated below \cite{CFGV02, Hua94}.

\begin{conjecture}[Rota's Basis Conjecture]
Let $M$ be a matroid of rank $n$, and suppose $B_1,B_2,\ldots B_n$ are $n$ bases of $M$. Then for each $i$, there is a linear order of $B_i$, say $B_1=\{a_1,a_2,\ldots,a_n\}; B_2=\{b_1,b_2,\ldots,b_n\}; \ldots B_n=\{c_1,c_2,\ldots,c_n\}$; such that
$C_1 = \{a_1, b_1, \ldots ,c_1\}$; $C_2 = \{a_2,b_2,\ldots,c_2\}$; \ldots; $C_n = \{a_n,b_n,\ldots,c_n\}$ are bases of $M$.
\end{conjecture}

In the special case where $M$ is a representable matroid, Rota's Basis Conjecture states that given an $n \times n$ array of boxes and $n$ disjoint bases of an $n$-dimensional vector space, we can place a distinct element of the $i^{th}$ basis in each box of the $i^{th}$ row of the array, in such a way that each column of the array contains a basis.  First published in \cite{Hua94}, Rota's basis conjecture remains open.

Chow and Taylor hoped to prove Rota's Basis Conjecture by using induction on the number of boxes of a Young diagram to obtain the statement for squares.  While this strategy was ultimately unsuccessful, it paved the way for Chow, Fan, Goemans, and Vondr\'{a}k to investigate the combinatorics of Latin tableaux in their 2003 paper \cite{CFGV02}.

In \cite{CFGV02}, the authors attempt to prove a version of Rota's conjecture for Young diagrams in the case of \emph{free matroids}--matroids in which all sets are independent \cite{CFGV02}.  Here, Rota's Basis Conjecture reduces to a combinatorial statement about existence of Latin tableaux.  It is easily shown that a Latin tableaux of shape $\lm$ cannot exist unless $\lm$ satisfies a somewhat technical condition called \emph{wideness}.  Based on extensive computational evidence, the authors of \cite{CFGV02} made the following conjecture.

\begin{manualconjecture}{\ref{WPC}}[The Wide Partition Conjecture]
    A partition $\lm$ is wide if and only if there exists a  Latin tableaux of shape $\lm$.
\end{manualconjecture}

The Wide Partition Conjecture (WPC) is still open, although some special cases are proved in \cite{CFGV02}.  The conjecture may be re-stated in a number of contexts, including list colorings and network flows \cite{CFGV02}.  D{\"u}rr and Gu{\'i}{\~n}ez gave a reformulation of the WPC from the perspective of discrete tomography, the study of reconstructing geometric objects from their projections to finite sets of integers \cite{DG13}. More specifically, they translated the Wide Partition Conjecture into a claim about whether certain combinations of integer vectors can be realized as projections of finite subsets of a discrete lattice. 

Huang and Rota proved in \cite{Hua94} that for even $n$, Rota's basis conjecture is equivalent to the Alon-Tarsi conjecture, a purely combinatorial statement about the number of ``even'' versus ``odd'' Latin squares of order $n$ \cite{Hua94, AT92}. The Alon-Tarsi conjecture is known to be true for infinitely many values of $n$. In particular, for $p$ an odd prime, the Alon-Tarsi conjecture was proved for $n = p+1$ by Drisko in the late 1990's, and for $n = p - 1$ by Glynn more than a decade later \cite{Dri97, Gly10}. 

\begin{remark}As an historical note, we mention that Drisko published an incorrect proof of the Alon-Tarsi conjecture for the case $n = 2^r p$, where $p$ is an odd prime and $r$ is a positive integer \cite{Dri98}. Drisko's work relied on a result from the literature, which was later found to be unreliable \cite{Gly10}. While the main theorem of \cite{Dri98} is not proved, other results from the paper remain valid and interesting. For more, see \cite{SW12}.
\end{remark}

A suitable analogue of the Alon-Tarsi conjecture for Latin tableaux would likely imply a tableaux analogue of Rota's basis conjecture, for matroids representable over a finite field \cite{CFGV02}.  Hence understanding the combinatorics of Latin tableaux may yield new insights into long-standing algebraic questions.
\color{black}

\subsection{Main Results}
Our objective in this paper is two-fold.  We seek both to investigate the structural properties of isotopy graphs of Latin tableaux, and to characterize Latin tableaux with isotopy graphs that are isomorphic to cubes. The second objective is achieved in Section \ref{section:cube}, with Theorems \ref{cube_criterion} and \ref{theorem:cube}.

We say a pair of columns in a Latin tableau is \emph{symmetric} if interchanging the two columns has the same effect as interchanging some pair of entries.
\begin{manualtheorem}{\ref{cube_criterion}}
    For $T$ a Latin tableau of shape $\lm$, the graph $\mathscr{G}(T)$ is a cube if and only if both of the following hold.
    \begin{enumerate}
        \item The shape $\lm$ has no more than two rows and no more than two columns of the same length.
        \item The tableau $T$ has no nontrivial symmetries, except those arising from symmetric pairs of columns.
    \end{enumerate}
\end{manualtheorem}

\begin{manualtheorem}{\ref{theorem:cube}}
    For every positive integer $d$, there exists a Latin tableau $T$ such that the isotopy graph of $T$ is isomorphic to a $d$-dimensional cube.
\end{manualtheorem}

Earlier sections of the paper deal with structural properties of general isotopy graphs. In Section \ref{section:clique}, we investigate clique numbers of isotopy graphs, and prove the following theorem.

\begin{manualtheorem}{\ref{clique_number}}
    Let $\lm$ be a wide partition such that $\mathscr{G}(\lm)$ is nonempty.  The clique number of $\mathscr{G}(\lm)$ is either 1, 2 or 4.
\end{manualtheorem}

Moreover, we show that ``most'' isotopy graphs have clique number two.  In particular, most isotopy graphs are triangle-free. Proposition \ref{prop:triangle} gives a characterization of Latin tableaux whose isotopy graphs contain a triangle which, by Theorem \ref{clique_number}, will be always part of a 4-clique.

In Section \ref{section:degree}, we give a formula for the degree of each vertex of an isotopy graph, based on both the shape and filling of a Latin tableau $T$. 

\begin{manualtheorem}{\ref{count_pairs}}
    Let $T$ be a young diagram of shape $\lm$.  Let $a$ be the number of pairs of rows of $\lm$ that have the same length, let $b$ be the number of pairs of columns of $\lm$ that have the same length, and let $p$ be the number of symmetric pairs of columns in $T$, then the degree of a vertex in $\mathscr{G}(T)$ is 
    \begin{equation}\label{formula:degree}
    \text{deg}(\mathscr{G}(T)) = a + 2b - p.
    \end{equation}
\end{manualtheorem}

The formula above will be useful in Section \ref{section:cube}, where we characterize isotopy graphs isomorphic to cubes.

\subsection*{Acknowledgements} We are grateful to the anonymous referee for their helpful suggestions, which markedly improved the exposition in this paper, and for pointing us toward the references \cite{Gly10}, \cite{HKO11}, and \cite{SW12}. We thank Timothy Chow for enlightening conversations about Latin tableaux.
\section{Preliminaries}\label{preliminaries}

\subsection{Background on Latin tableaux}

\begin{definition}
A \emph{partition of $n$} is a tuple of positive integers 
\[\lm = (\lm_1,\ldots,\lm_k)\] where 
\[\lm_1 + \lm_2 + \cdots + \lm_k = n.\]
\end{definition}

\begin{definition}
Let $\lambda=(\lambda_1, \lambda_2, ..., \lambda_n)$ be a partition of $n$, written so that the entries $\lm_i$ are non-increasing.  A \textit{Young diagram} of shape $\lm$ is an array of boxes that are arranged in left-justified rows, where the $i^{th}$ row has $\lm_i$ boxes. We will often use the symbol $\lm$ to denote both the tuple and the diagram with shape $\lm.$

We number the rows of a Young diagram from top to bottom, and the columns from left to right.  We let $(i,j)$ denote the position of the box in row $i$ and column $j$.
\end{definition}

\begin{definition}
The \emph{transpose} of a Young diagram $\lm$, denoted $\lm'$, is the Young diagram obtained by interchanging the rows and columns of $\lm.$ 
Equivalently, $\lm$ has a box $(i,j)$ if and only $\lm'$ has a box $(j,i)$.

The transpose of a partition $\lm$ is the partition that gives the \emph{shape} of the transpose of the Young diagram with shape $\lm.$
\end{definition}

\begin{definition}
A \emph{Young tableau} (or simply \emph{tableau}) is an assignment of a positive integer to each box in a Young diagram.  
A \textit{Latin tableau} is a Young tableau of shape $\lambda=(\lambda_1, \lambda_2, ..., \lambda_n)$
such that
\begin{itemize}
    \item Row $i$ of the tableau contains the numbers in $[\lambda_i]$ in some order.
    \item No number appears more than once in the same column.
\end{itemize}
\end{definition}

\begin{definition}
A \emph{Latin square} is a Latin tableau whose Young diagram has a shape of a square.  A \emph{Latin rectangle} is a Latin tableau whose Young diagram has a shape of a rectangle.
\end{definition}

\begin{definition}
The \textit{content} of a Young tableau $T$, denoted $\text{cont}(T)$, is the sequence $\mu = (\mu_1,\mu_2,\ldots )$, where $\mu_i$ is the number of times the symbol $i$ appears in $T$.
\end{definition}

\begin{remark}\label{Rem:entries}
    For $T$ a Latin tableau of shape $\lm$, the content of $T$ is the partition $\lm'.$  Indeed, the number of entries $i$ in $T$ is precisely the number of rows of $T$ with length greater than or equal to $i$, which is precisely the length of row $i$ of the partition $\lm'.$
\end{remark}

We note that Latin tableaux only exist for some shapes $\lm$. Chow, Fan, Goemans and Vondrak conjectured that the existence of a Latin tableau of shape $\lm$ depends on whether $\lm$ is \emph{wide}, a somewhat technical condition which will be defined below \cite{CFGV02}.

\begin{conjecture} \cite{CFGV02}\label{WPC}
    A Young diagram $\lm$ is wide if and only if there exists a Latin tableau of shape $\lambda$.
\end{conjecture}

 It is easy to check that wideness is necessary; sufficiency, however, remains open. The next few definitions are needed to define a wide partition, which we will do at the end of this section.

\begin{definition}
Let $\lm$ and $\mu$ be partitions of $n$, and write 
\[\lm = (\lm_1, \lm_2, \ldots, \lm_j)\]
\[\mu = (\mu_1, \mu_2, \ldots, \mu_k)\]
where the entries of $\lm$ and $\mu$ are non-increasing.
We say $\lm \leq \mu$ in the \emph{dominance order} on partitions of $n$ if for each $i \in \mathbb{N}$ we have
\[\lm_1 + \lm_2 + \ldots + \lm_i \leq \mu_1 + \mu_2 + \ldots + \mu_i\]
where we take $\lm_{\ell} = 0$ for $\ell > j$, and $\mu_{\ell} = 0$ for $\ell > k$.
\end{definition}

\begin{definition}
Let
\[\lm = (\lm_1, \lm_2, \ldots, \lm_k)\]
be a partition.  We say $\mu \subseteq \lm$ if
$\mu$ is a partition obtained by deleting some of the entries of $\lm$. 
\end{definition}

\begin{definition}
A Young diagram $\lm$ is \emph{wide} if for each $\mu \subseteq \lm$, we have $\mu \geq \mu'$ is dominance order.
\end{definition}

\begin{figure}
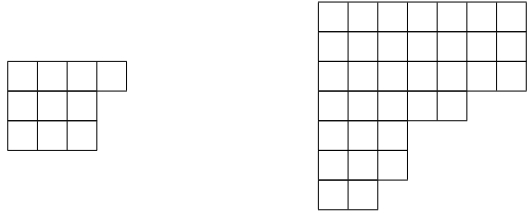

    \centering
    \[\yng(4,3,3) \hspace{1in} \yng(7,7,7,5,3,3,2)\]
    \caption{Examples of wide partitions.}
    \label{fig:wide}
    \color{black}
\end{figure}

\begin{remark}
A $m \times n$ rectangle with $m \leq n$ is a wide partition, as is any \emph{staircase shape}--that is, any partition where successive rows decrease in length by one.
\end{remark}

\subsection{Isotopies and isotopy graphs}

\begin{definition}
We say that two Latin tableaux $T_1$ and $T_2$ are \emph{isotopic} if $T_1$ can be obtained from $T_2$ by performing a finite sequence of the following elementary transformations:
\begin{enumerate}
    \item Permuting two rows of $T_2$ which have the same length.
    \item Permuting two columns of $T_2$ which have the same length.
    \item If integer $i$ and $j$ appear the same number of times in the filling of $T_2$, replace each $i$ with a $j$ and vice versa.
\end{enumerate}
We say two elementary transformations are of the same \emph{type} if they both act on rows, both act on columns, or both act on entries.
\end{definition}
\begin{definition}
An \emph{isotopy} of a Latin tableau $T$ is a composition of some sequence of the above three transformations. The \emph{isotopy class} of $T$ is the set of all Latin tableaux which are isotopic to $T$.
\end{definition}

\begin{definition}
Let $T$ be a Latin tableau with shape $\lambda$.  We define the \emph{isotopy graph} $\mathscr{G}(T)$ of $T$ as follows.
\begin{enumerate}
    \item The vertices of $\mathscr{G}(T)$ are Latin tableaux of shape $\lm$ that are isotopic to $T$.
    \item Two vertices in $\mathscr{G}(T)$ are adjacent if one can be obtained from the other by a single elementary transformation.
\end{enumerate}
We label the edge of $\mathscr{G}(T)$ which connects vertices $u$ and $v$ by the elementary transformation which relates $u$ and $v$. 
\end{definition}

\begin{definition}
The \emph{isotopy graph} $\mathscr{G}(\lm)$ of a Young diagram $\lambda$ is the disjoint union of the isotopy graphs $\mathscr{G}(T)$ as $T$ ranges through a set of representatives for the isotopy classes of Latin tableaux with shape $\lm$. 
\end{definition}

\begin{figure}[ht]
    \centering
    \includegraphics[trim=3in 5in 3in 1.75in, clip] {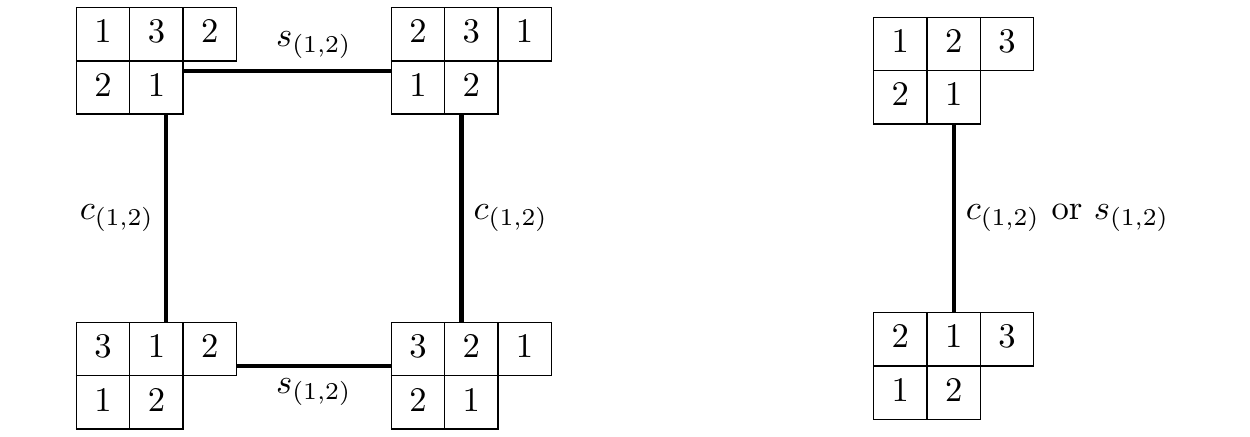}
    \caption{
    Isotopy graph for $\lm = (3,2)$. The transformation $c_{(1,2)}$ permutes the first two columns, while $s_{(1,2)}$ permutes entries $1$ and $2$. These two components of the isotopy graph are cubes of dimension 1 and 2.}
   \label{fig:small_graph}
\end{figure}

A major goal of this project is to characterize the graphs $\mathscr{G}(T)$ for various Latin tableaux $T$, and construct tableaux $T$ for which $\mathscr{G}(T)$ has an especially nice structure. We note that the shape $\lm$ of the tableaux $T$ plays a key role in determining the structure of $\mathscr{G}(T).$

\begin{definition}
    For a Latin tableau $T$, let $r_{(i,j)}$ denote the elementary transformation that permutes the rows $i$ and $j$.  Similarly, let $c_{(i,j)}$ denote the elementary transformation that permutes columns $i$ and $j$, and let $s_{(i,j)}$ denote the elementary transformation that permutes the entries $i$ and $j$ wherever they appear in the tableau. 
\end{definition}

\begin{remark}
    Note that the transformation $r_{(i,j)}$ is defined for a tableau $T$ if and only if rows $i$ and $j$ have the same length.  Similarly $c_{(i,j)}$ is defined if and only if columns $i$ and $j$ have the same length, and $s_{(i,j)}$ is defined if and only if entries $i$ and $j$ appear an equal number of times in the tableau.
\end{remark}

\begin{definition}\label{def:group}
    Let $T$ be a Latin tableau of shape $\lm$, and let $a_1 > a_2 > \cdots > a_b$ be the \emph{distinct} integers which appear as parts of the partition $\lm$.  For $1 \leq i \leq b$, let $k_i$ be the number of times $a_i$ appears in $\lm$.  Then the \emph{row group} of $T$ is the direct product 
    $S_{\text{row}}(T) = S_{k_1} \times S_{k_2} \times \cdots S_{k_b}$
    where $S_{k_i}$ is the symmetric group on $[k_i]$.We can define the \emph{column group}, denoted by  $S_{\text{col}}(T)$, analogously, with $\lm'$ taking the place of $\lm$.  Finally, by Remark \ref{Rem:entries}, we know that $\text{cont}(T) = \lm'$.  Hence we define the \emph{entry group} $S_{\text{ent}}(T)$ in the same way, again letting $\lm'$ take the place of $\lm$.
\end{definition}

We note that $S_{\text{row}}(T)$ is the group of isotopies of $T$ which are obtained by permuting rows of the same length; for each $i$, the factor $S_{k_i}$ acts by permuting the $k_i$ rows of length $a_i$.  The group $S_{\text{row}}(T)$ is generated by the set of elementary transformations $r_{(i,j)}$ defined above. The same kind of remark applies to $S_{\text{col}}(T)$ and $S_{\text{ent}}(T)$.

\begin{definition}
    The \emph{isotopy group} of a Latin tableau $T$ is the direct product 
    \[\mathscr{S}(T) = S_{\text{row}}(T) \times S_{\text{col}}(T) \times S_{\text{ent}}(T).\]
\end{definition}

It is clear that transformations which are not of the same type commute, 
so we have a well-defined, transitive action of $\mathscr{S}(T)$ on the isotopy class of $T$ by isotopies of the Latin tableaux.  We adopt the convention that elements of $\mathscr{S}(T)$ act on the left, so for $u, v \in \mathscr{S}(T)$, $vuT$ is the tableau obtained from $T$ by applying first the isotopy $u$, then the isotopy $v$.

\begin{remark}\label{coset}
Let $T$ be a Latin tableaux, and let $H$ be the stabilizer of $T$ in the group $\mathscr{S}(T)$. The isotopy graph $\mathscr{S}(T)$ is closely related to the \emph{Schreier coset graph} of $\mathscr{S}(T)/H$, whose vertices are left cosets of $H$, and whose edges are of the form $uH \text{---} vuH$ where $u \in \mathscr{S}(T)$ and $v$ is an elementary transformation. However, in the case where two elementary transformations $v$ and $v'$ act identically on the tableau $uT$, the isotopy graph has a single edge from $uT$ to $vuT$, which is labeled with both $v$ and $v'$.  Hence the isotopy graph $\mathscr{G}(T)$ is obtained from the Schreier coset graph of $\mathscr{S}(T)/H$ by replacing each set of multiple edges with a single edge.   We feel this convention is more natural for our purposes, as we are more interested in the combinatorics of the tableaux than in the structure of the group $\mathscr{S}(\lm).$

Every Schreier coset graph is \emph{vertex transitive}, meaning that for two vertices $x,y$ of the graph, there is an automorphism of the graph that carries $x$ to $y$.  Hence $\mathscr{S}(T)$ is vertex-transitive for every Latin tableaux $T$.
\end{remark}

It follows from Definition \ref{def:group} that each factor in the product $\mathscr{S}(T)$ depends only on the shape $\lm$ of $T$.  Hence we may speak of the row group of the partition $\lm$, $S_{\text{row}}(\lm)$, and similarly for the column and entry groups. When there is no ambiguity about which Latin tableaux we are referring to, we sometimes use $S_{\text{row}}$ instead of $S_{\text{row}} (T)$, and similarly for the column and entry groups.

\begin{definition}
    The \emph{isotopy group} of a partition $\lm$ is the direct product 
    \[\mathscr{S}(\lm) = S_{\text{row}}(\lm) \times S_{\text{col}}(\lm) \times S_{\text{ent}}(\lm).\]
\end{definition}

\begin{lemma}
Given a Young diagram $\lm$ which admits a Latin filling, the group $\mathscr{S}(\lm)$ acts on $\mathscr{G}(\lm)$. The orbits of this action are the graphs $\mathscr{G}(T)$, as $T$ ranges over a collection of representatives for the isotopy classes of Latin tableaux of shape $\lm$.
\end{lemma}

\begin{remark}
    If a Young diagram $\lm$ does not admit a Latin filling, the isotopy graph of $\lm$ is the empty graph.  Hence by Conjecture \ref{WPC}, it is believed that the isotopy graph of $\lm$ is nonempty if and only if $\lm$ is wide.
\end{remark}

\section{Clique numbers of isotopy graphs}
\label{section:clique}

\begin{definition}
A clique in a graph is a set of vertices such that any two of them are connected by an edge. The clique number of a graph is the size of a largest clique.
\end{definition}

\begin{figure}
    \centering
    \includegraphics[trim=3in 3.75in 3in 1.75in, clip]{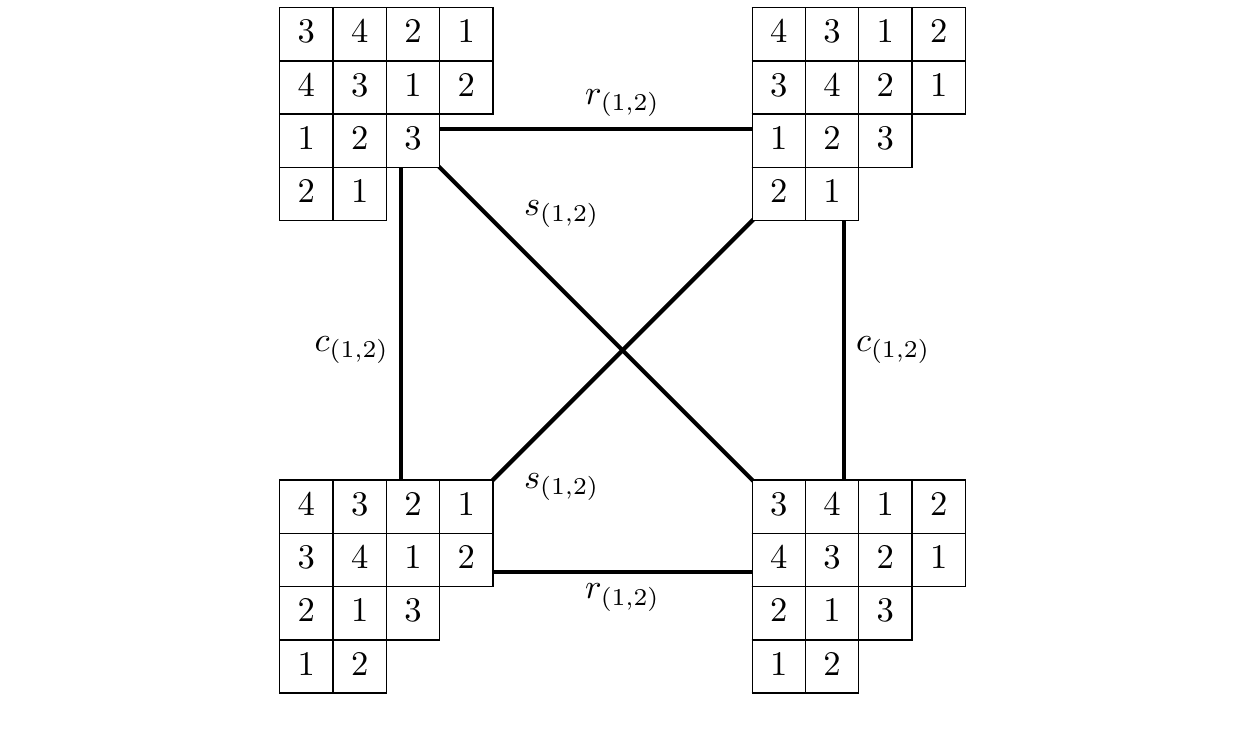}
    \caption{An isotopy graph which is a complete graph on four vertices.}
    \label{fig:clique}
\end{figure}
In this section, we study clique numbers of isotopy graphs.  We show that isotopy graphs of Young diagrams have clique number $1,2$ or $4$, and classify the Young diagrams whose isotopy graphs have clique number $4$. In particular, we show in Proposition \ref{prop:triangle} that the isotopy graphs of most Young diagrams have clique number $2$.  
\begin{lemma}
\label{different_sides}
Let $\lm$ be a wide partition other than $(4,4)$.  Then $\mathscr{G}(\lm)$ has clique number at most 4. In a 3-cycle of $\mathscr{G}(\lm)$, no two edges are labeled with elementary transformations of the same type.  
\end{lemma}

\begin{proof}
    Let $u$, $v$ and $w$ be distinct vertices of $\mathscr{G}(\lm)$.  
    Suppose $u = r(v)$, $w = r'(v)$ for distinct elementary transformations $r, r' \in S_{\text{row}}$.  We claim that $u$ and $w$ are not adjacent vertices of $\mathscr{G}(\lm)$.  
    
    So see this, note that $r'$ is an involution, so we have $v = r'(w)$ and $u = rr'(w)$.  Since $rr'$ is a product of two transpositions, $rr'$ either cyclically permutes three rows of $w$, or transposes two distinct pairs of rows.  In either case, $rr'$ acts on at least three rows, so $u$ and $w$ cannot be related by a single row transposition.
    
    In a wide partition, there cannot be more than two rows of length two.  Hence $rr'$ acts on at least two rows of length three.  In a Latin tableau, acting by a row permutation means replacing one row with another that differs from the first in every box.  Hence boxes in at least three columns are  affected; as are at least three distinct entries.  It follows that $u$ and $w$ are not related by a single transposition of columns or entries.
    
    We have shown that no triangle in $\mathscr{G}(\lm)$ has two sides both labeled by row transpositions. Next, let $u$, $v$ and $w$ be distinct vertices of $\mathcal{G}(\lm)$, and suppose $u = c(v)$, $w = c'(w)$ for $c, c'$ elementary transformations in $S_{\text{col}}$.  As above, we have $u = cc'(w)$. As above, the transformation $cc'$ changes at least three distinct columns and three distinct entries, so $u$ and $w$ cannot be related by transposing a single pair of columns or entries.  
    Moreover, $cc'$ changes entries in at least three rows, unless $cc'$ acts on three of four columns of length two.  Suppose this is the case.  We must check that $cc'$ does not act on $w$ by switching the two rows which intersect these columns.  This is clear unless $\lm$ is  either a $2 \times 3$ or $2 \times 4$ rectangle, since otherwise some boxes of the affected rows will be fixed by $cc'$, while others are moved.  A check of $\mathscr{\lm}$ for the $2 \times 3$ rectangle reveals that no triangle with two sides labeled by column transformations exists in this case.
    The case of a $2 \times 4$ rectangle is not addressed by this lemma, and will be handled separately.
    
    Finally, suppose $u$, $v$ and $w$ are distinct vertices of $\mathscr{G}(\lm)$, and suppose $u = s(v)$, $w = s'(w)$ for $s, s'$ elementary transformations in $S_{\text{ent}}$.  As above, we have $u = ss'(w)$.  Again $ss'$ must permute at least three distinct entries, and so cannot act by a transposition of two entries.  Moreover, in a Latin tableaux, all entries must appear in the first (longest) row.  Hence $ss'$ permutes at least three boxes in the same row, and cannot act by a single transposition of columns.  Finally, we must check that $ss'$ cannot act by switching two rows.  Again, this is manifestly impossible unless $\lm$ is a $2 \times 3$ or $2 \times 4$ rectangle.  A check reveals that no such triangle exists in $\mathscr{G}(\lm)$ when $\lm$ is a $2 \times 3$ rectangle.
    
    We have shown that for $\lm \neq (4,4)$, each triangle in $\mathscr{G}(\lm)$ must have sides labeled with elementary transformations of distinct type.  Now, suppose $\lm \neq (4,4)$ is a wide partition such that $\mathscr{G}(\lm)$ has clique number greater than two. Let $u$ be any vertex of $\mathscr{G}(\lm)$ which is contained in a clique of size greater than two, and let $v$ and $w$ be additional vertices of this clique. Then the edges $uv$ and $uw$ form two sides of a triangle in the clique, and must therefore be labeled by elementary transformations of distinct type.  Since there are only three types of elementary transformations, it follows that any clique in $\mathscr{G}(\lm)$ must have at most four vertices, so the clique number of $\mathscr{G}(\lm)$ is at most four.
\end{proof}

For $\lm = (4,4)$, the statement that no triangle in the isotopy graph has two sides labeled by transformations of the same type no longer holds, as shown in Figure \ref{fig:four_clique}.  However, by modifying our arguments slightly, we can show that $\mathscr{G}(\lm)$ has clique number four in this case as well.

\begin{figure}
    \centering
    \includegraphics[trim=3in 2.75in 3in 1.5in, clip]{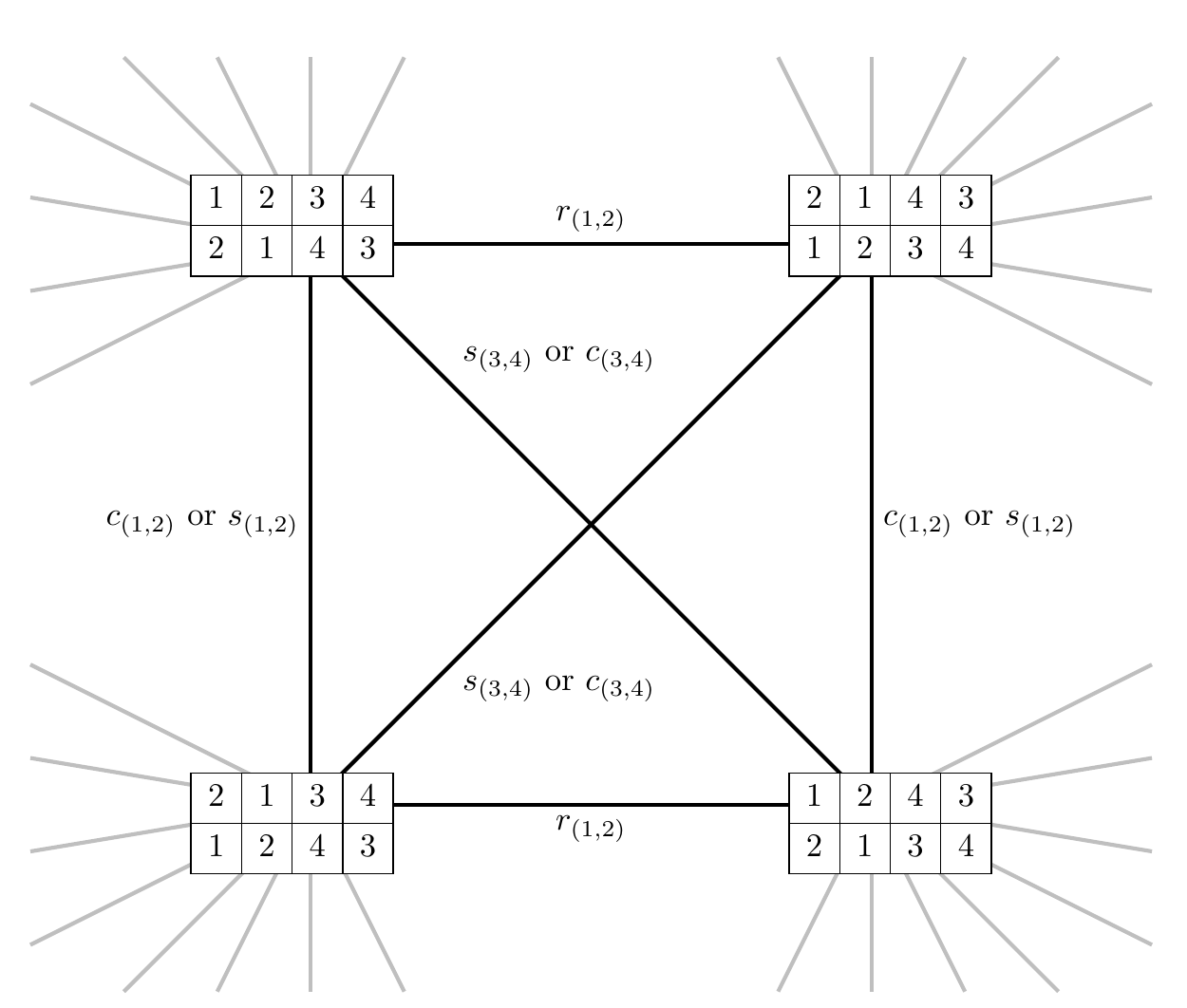}
    \caption{Four fillings of $\lm=(2,2)$, which form a maximal four-clique in the isotopy graph.}
    \label{fig:four_clique}
\end{figure}

\begin{lemma}\label{two_by_four}
Let $\lm = (4,4)$. The clique number of $\mathscr{G}(\lm)$ is $4$.
\end{lemma}

\begin{proof}

Let $T$ be a filling of $\lm$, and consider a triangle $C$ in $\mathscr{G}(T)$.  Note that a composition of two column operations must impact at least three different columns, and at least three different entries.  Hence there cannot be a triangle in $\mathscr{G}(T)$ which has three sides labeled by column transformations; or two sides labeled by column operations and one side labeled by a permutation of entries.  Similarly, no triangle has all three sides labeled by transformations on entries, or two labeled by entry transformations and one labeled by a column transformation.
Hence, the possible labels for the sides of the triangle are 
$\{r,c,s\}$, $\{r,c,c\}$ and $\{r,s,s\}$.
We claim this implies that $\mathscr{G}(\lm)$ has a clique number of at most $4$.

To see this, note that there is only one row transformation $r = r_{(1,2)}$ in the isotopy group of $\lm$.  Hence each vertex of $\mathscr{G}(\lm)$ is incident at exactly one edge labeled $r$.  Let $u,v,w,x$ be vertices of a four-clique in $\mathscr{G}(\lm)$. Any three of these four vertices form a triangle, which has one side labeled $r$.  

Suppose $v = r(u).$  Then since $\{u,w,x\}$ form the vertices of a triangle, and $u$ is incident at exactly one edge labeled $r$, we must have that $w = r(x)$.  Hence in the subgraph induced by $\{u,v,w,x\}$, each vertex is incident at an edge labeled $r$.  Let $y$ be another vertex of $T$.  Then no edge from $y$ to a vertex in $\{u,v,w,x\}$ can be labeled with $r$. If $y$ were adjacent to all the vertices $\{u,v,w,x\}$, then we would have a triangle with vertices $\{x,y,u\}$.  Since the edge from $x$ to $u$ must have a label other than $r$, no edge in this triangle can be labeled $r$, a contradiction. So the clique number of $\mathscr{G}(\lm)$ is at most four.

\end{proof}

\begin{theorem}
    \label{clique_number}
    Let $\lm$ be a wide partition such that $\mathscr{G}(\lm)$ is nonempty.  The clique number of $\mathscr{G}(\lm)$ is either 1, 2 or 4.
\end{theorem}

\begin{proof}
    We have seen that the clique number of $\mathscr{G}(\lm)$ is $4$ when $\lm = (4,4)$.  Suppose now that $\lm$ is not equal to $(4,4)$.
    
    By Lemma \ref{different_sides}, it suffices show that if $\mathscr{G}(\lm)$ has a 3-cycle, that 3-cycle must be contained in a clique of size 4. Suppose $\mathscr{G}(\lm)$ has a 3-cycle $C$. By Lemma \ref{different_sides}, the edges of $C$ must be labeled by distinct elementary transformations $r,c$ and $s$ where $r \in S_{\text{row}}$, $c \in S_{\text{col}}$ and $s \in S_{\text{ent}}.$  
    
    Let $u, v$ and $w$ be the vertices of $C$, and suppose without loss of generality that we have $v = r(u)$, $w = s(u)$, and $w = c(v)$.  Let $x = c(u)$.  Since $r$ and $c$ pairwise commute, we have 
    \[w = cr(u) = rc(u) = r(x).\]
    So $x$ is adjacent to $w$.  
    
    Since $c(v) = w = s(u)$, and $c$ is an involution, we have
    \[v = c(w) = cs(u)\]
    and since $c$ and $s$ commute we have
    \[v = sc(u) = s(x).\]
    So $x$ is adjacent to $v$, and $x$ is adjacent to all three vertices of $C$.  Hence $\{u,v,w,x\}$ are the vertices of a clique of size 4 in $\mathscr{G}(\lm)$. This completes the proof.
\end{proof}

\begin{proposition}{\label{prop:triangle}}
Let $T$ be a Latin tableaux of shape $\lm$.  Then $\mathscr{G}(T)$ has a triangle if and only if there exist elementary transformations $r= r_{(i_1,i_2)} \in S_{\text{row}}$, $c=c_{(j_1,j_2)} \in S_{\text{col}}$ and $s=s_{(a_1,a_2)} \in S_{\text{ent}}$ such that $a_1,a_2 \in [4]$, $rc(T) = s(T)$ and all of the three following properties hold:
\begin{enumerate}
    \item \label{outside_square} All boxes which are in rows $i_1$ or $i_2$ but \emph{not} in columns $j_1$ or $j_2$ have entries in $\{a_1,a_2\}$.  All boxes which are in columns $j_1$ or $j_2$ but \emph{not} in rows $i_1$ or $i_2$ also have entries in $\{a_1,a_2\}$.
    \item \label{fixed_by_s}No entry in $\{a_1,a_2\}$ appears in any other row or column of $T$.
    \item \label{small_square} Suppose rows $i_1$ and $i_2$ intersect columns $j_1$ and $j_2$, so that $T$ has boxes in positions $(i_1,j_1)$, $(i_1,j_2)$, $(i_2,j_1)$ and $(i_2,j_2)$, which mark the corners of a square within $T$.  Then the entries in these four boxes form one of the patterns shown below, where $i,j, k \in \{1,2\}$ and \[\{b_1,b_2\} = [4] \backslash \{a_1,a_2\}.\]
    
    \[\young(\bi\bj,\bj\bi) \qquad \young(\ai\bk,\bk\aj) \qquad \young(\bk\ai,\aj\bk)\]
    
\end{enumerate}
\end{proposition}

Note that from the conditions stated in the proposition, it follows that rows $i_1$ and $i_2$ both have length at most $4$, and the same holds for columns  $j_1$ and $j_2.$--see Appendix \ref{Appendix}.

\begin{proof}
    The fact that this holds for $\lm = (4,4)$ follows from the proof of Lemma \ref{two_by_four}.  So assume $\lm \neq (4,4)$. 
    
    First suppose that $T$ satisfies the above criteria.  Let $r = r_{(i_1,i_2)}$ and $c = c_{(j_1,j_2)}$.  If rows $i_1$ and $i_2$ intersect columns $j_1$ and $j_2$, then $rc$ interchanges boxes $(i_1,j_1)$ and $(i_2,j_2)$; and also interchanges boxes $(i_1,j_2)$ and $(j_2,i_1)$. By Criterion \ref{small_square}, this means that $rc$ either interchanges pairs of boxes that contain the same entry; or pairs where the two boxes contain $a_i$ and $a_j$ respectively with $i,j \in \{1,2\}$.   Hence $rc$ acts on these four boxes by applying $s_{(a_1,a_2)}$. 
    
    Next, note that $rc$ swaps any remaining boxes in row $i_1$ with the box in the same column of row $i_2$.  Since the boxes in question must contain entries $a_i,a_j$ for $i,j \in \{1,2\}$ by criterion \ref{outside_square}, and no two boxes in the same column may have the same entries, it follows that $rc$ acts on these rows by applying $s_{(a_1,a_2)}$. Similarly, $rc$ acts by $s_{(a_1,a_2)}$ on the boxes that are in columns $j_1,j_2$ but not rows $i_1,i_2$.  Hence $rc$ acts on the entire tableau by $s_{(a_1,a_2)}$ and the isotopy graph contains a triangle.
    
    Next, suppose the isotopy graph of $T$ has a triangle.  Note that the isotopy group of $\lm$ acts transitively on vertices of $\mathscr{G}(T)$.  Hence $\mathscr{G}(T)$ has a triangle if and only if there is a triangle which includes the vertex indexed by $T$ itself.  
    
    By Lemma \ref{different_sides}, $T$ is a vertex of a triangle if and only if there are elementary transformations $r \in S_{\text{row}},$ $c \in S_{\text{col}}$, $s \in S_{\text{ent}}$ such that $rc(T) = s(T)$.  Suppose $r = r_{(i_1,i_2)}$, $c = c_{(j_1,j_2)}$, and $s = s_{(a_1,a_2)}$.
    
    Then in particular rows $i_1$ and $i_2$ are of the same length, and so are columns $j_1$ and $j_2$. All entries which are in rows $i_1$ or $i_2$, but not in columns $j_1$ or $j_2$, are changed by the action of $rc$.  Hence those boxes must contain entries in $\{a_1,a_2\}$. Note that this implies rows $i_1,i_2$ each have length at most 4, since each contains at most two distinct entries in boxes outside of columns $j_1,j_2$.  So $a_1,a_2 \in [4]$ as desired.  Similarly, all entries in columns $j_1$ or $j_2$ which are not contained in $i_1$ or $i_2$ must have entries in $\{a_1,a_2\}$ and Criterion \ref{outside_square} holds.  Moreover, $T$ cannot have entries from $\{a_1,a_2\}$ in boxes which are not moved by $rc$, so Criterion \ref{fixed_by_s} holds. 
    
    Finally, if rows $i_1$ and $i_2$ intersect columns $j_1$ and $j_2$, then again $rc$ acts by interchanging boxes $(i_1,j_1)$ and $(i_2,j_2)$; and interchanging boxes $(i_1,j_2)$ and $(j_2,i_1)$.  This action is only equivalent to applying $s_{(a_1,a_2)}$ if the entries in these four boxes fall into one of the three cases in Criterion \ref{small_square}.  This completes the proof.
\end{proof}

By Proposition \ref{prop:triangle}, we can classify the Young diagrams $\lm$ such that $\mathscr{G}(\lm)$ has a triangle. In any such $\lm$ there must be a pair of rows $i_1,i_2$ of the same length, which is at most $4$; and a pair of columns $j_1,j_2$ of the same length, which is at most $4$, as in the statement of Proposition \ref{prop:triangle}.  If rows $i_1,i_2$ do not intersect columns $j_1,j_2$, then Condition 1 from Proposition \ref{prop:triangle} implies that both rows $i_1,i_2$ and columns $j_1,j_2$ are of length $2$, and both contain only entries in $\{1,2\}$.  Hence $\lm$ is a wide partition with four rows, such that rows $3$ and $4$ are of length $2$; and rows $1$ and $2$ are of length at least $4$.

If rows $i_1,i_2$ and columns $j_1,j_2$ intersect, this again imposes a substantial restriction on the shape $\lm$. We note, however, that the existence of a pair of columns and a pair of rows that intersect in the necessary way does not guarantee that $\mathscr{G}(\lm)$ has a triangle.  The exceptions are wide partitions with 4 or 5 rows, whose lower 4 rows form a copy of the diagram
\[\lm = \yng(4,4,3,1).\]
To see why, let $T$ be a Latin filling of $\lm$. Note that if $\mathscr{G}(T)$ were to have a triangle, it would in particular have a triangle containing the vertex corresponding to $T$ itself.  This triangle would necessarily have two sides labeled $r_{(1,2)}$ and $c_{(2,3)}.$  The third side would be labeled by $s_{(a_1,a_2)}$, where $a_1$ and $a_2$ are the entries in the second and third boxes in row 3.  However, the last row of $\lm$ contains only a single box, which must be filled with a $1$.  Hence either $a_1$ or $a_2$ must equal $1$, which violates Condition 2 of Proposition \ref{prop:triangle}. Thus $\mathscr{G}(\lm)$ is triangle-free.

We include as an appendix a complete list of Young diagrams $\lm$ such that $\mathscr{G}(\lm)$ has a triangle.

\section{Vertex Degrees of Isotopy Graphs}
\label{section:degree}

By Remark \ref{coset}, each connected component of an isotopy graph is $k$-regular for some $k$, meaning all vertices have the same degree.  The lemma below gives a bound on the maximum degree of a vertex in an isotopy graph of a Latin tableau, which depends on both the shape of the tableau and its entries. 

Before stating the lemma, we give an example.  Figure \ref{fig:four_clique} shows four Latin tableaux of shape 
$\lm = (4,4)$, which form a maximal clique in the isotopy graph of the tableau
\[T = \young(1234,2143).\]
There are 13 elementary transformations that act on Latin tableaux of shape $\lm = (4,4)$: one transposition that acts on rows, 6 that act on columns, and 6 that act on entries. So the degree of any vertex of $\mathscr{S}(T)$ is at most 13. However, the transformation $c_{(1,2)}$, which switches the first two columns of $T$, acts identically to $s_{(1,2)}$, which switches entries $1$ and $2$. Similarly $c_{(3,4)}$, the transformation which switches the columns $3$ and $4$, acts identically to $s_{(3,4)},$ which switches entries $3$ and $4$.  Hence there are only $11$ elementary transformations which act distinctly on $T$, and the degree of each vertex of $\mathscr{G}(T)$ is $11$.

\color{black}

\begin{lemma}\label{degree}
    Let $T$ be a Young diagram of shape $\lm$.  Let $a$ be the number of \emph{pairs} of rows of $\lm$ which have the same length, and let $b$ be the number of \emph{pairs} of columns of $\lm$ which have the same length.  Then the degree $d_T$ of each vertex of $\mathscr{G}(\lm)$ is bounded above by $a+2b$.  Moreover the degree of each vertex in a connected component $\mathscr{G}(T)$ is equal to $a + 2b$ unless one of the following holds:
    \begin{enumerate}
        \item The first row of $\lm$ has length $n$, the second row has length $m < n$, and for some $i,j > m$, the boxes at positions $(1,i)$
 and $(1,j)$ contain entries $x,y \in \{m+1,\ldots,n\}$.
        \item The second row of $\lm$ has length $n$, the third row has length $m < n$, and there are columns $m < i,j \leq n$ whose boxes give a tableau of the form \[\young(xy,yx)\] for $\{x,y\} \subseteq \{m+1,\ldots,n\}$.
    \end{enumerate}
\end{lemma}

\begin{proof}
    The degree of a vertex in $\mathscr{G}(T)$ is the number of elementary transformations whose actions on $T$ are distinct. We have one elementary transformation for every pair of rows of the same length; one for every pair of columns which have the same length; and one for every pair of entries that appear the same number of times.  Since the number of pairs of entries that appear the same number of times is exactly the number of pairs of columns of the same length, the total number of elementary transformations in the isotopy group of $\mathscr{G}(\lm)$ is therefore $a + 2b$.  Hence $d_T = a + 2b$, as long as no two elementary transformations act identically on $T$.
    
    Suppose there exist two elementary transformations $u$ and $v$ which act identically on $T$.  We note that switching two rows of length 3 or more will change boxes in more than two distinct columns, and boxes with more than two distinct entries.  Hence no other elementary transformation can have the same effect as switching a pair of rows of length 3 or more.  Similarly, no other elementary transformation can have the same effect as switching two columns of length 3 or more, or two entries which each appear three or more times.  
    
    Suppose either $u$ or $v$ is a row transformation; without loss of generality $u$ acts by switching $2$ rows, which are of necessity both of length $2$ (a wide partition cannot have $2$ rows of length $1$). The transformation $v$ must act on precisely the same collection of boxes, by switching either $2$ columns or $2$ entries.  In the former case, $v$ can only move boxes which are also affected by $u$.  Hence $T$ has a pair of rows of length 2, which are wholly contained in a pair of columns of length 2, and $T$ is a $2 \times 2$ square, so condition (2) holds.  In the latter, we must have $v = s_{(1,2)}$.  But this means there are no entries $1$ or $2$ outside the two rows effected by $v$.  Hence those rows are the whole tableau, so $T$ is again a $2 \times 2$ square.
    
    Hence if $\lm$ is not a $2 \times 2$ square, it follows that neither $u$ nor $v$ acts on rows.  Suppose without loss of generality that $u$ acts by switching columns $i$ and $j$, and $v$ acts on entries.  Then the columns acted on by $u$ contain only entries which are acted on by $v$, and hence both have length either $1$ or $2$.  Suppose both have length $1$.  Then if the first row of $T$ has length $n$, the second has length $m < n$.  The entries acted on by $v$ must be contained only in the columns which are affected by $u$.  Since a row of length $k$ in a Latin tableau contains entries $\{1,2,\ldots,k\}$, it follows that the boxes in positions $(1,i)$ and $(1,j)$  which are affected by $u$ must contain entries in $\{x,y\} \subseteq \{n+1,\ldots,m\}$.  So Condition (1) holds.

    Similarly, suppose columns $i$ and $j$ both have length $2$.  Then there must be entries $\{x,y\}$ which appear in the boxes $(1,i), (1,j),(2,i),(2,j)$ and nowhere else in $T$. Since a row of length $k$ in a Latin tableau contains entries $\{1,2,\ldots,k\}$, this is impossible unless condition (2) holds.
\end{proof}

\begin{definition}
    If a Latin tableau $T$ satisfies either Condition (1) or Condition (2) from \ref{degree}, we call the two columns $i,j$ in question a \emph{symmetric} pair of columns.
\end{definition}

\begin{theorem}\label{count_pairs}
    Let $T$ be a young diagram of shape $\lm$.  Let $a$ be the number of pairs of rows of $\lm$ that have the same length, let $b$ be the number of pairs of columns of $\lm$ that have the same length, and let $p$ be the number of symmetric pairs of columns in $T$. Then the degree of a vertex in $\mathscr{G}(T)$ is $a + 2b - p$.
\end{theorem}

\begin{proof}
    We may partition the elementary transformations acting on $\mathscr{G}(T)$ into equivalence classes, by declaring two transformations equivalent if they act identically on $T$.  The degree of a vertex of $\mathscr{G}(T)$ is then the number of equivalence classes.  It follows from the proof of the previous lemma that two transpositions are equivalent if and only if they form a pair $u = c_{(i,j)}$, $v = s_{(x,y)},$ where columns $i$ and $j$ of $T$ form a symmetric pair of columns containing entries from $\{x,y\}$.  Hence we have $p$ equivalence classes containing two elementary transformations. The remaining $a + b - 2p$ transformations are equivalent only to themselves, giving a total of $a + 2b - p$ equivalence classes.
\end{proof}

In the case where $\lm$ is a square, the formula in Theorem \ref{count_pairs} has an especially nice form.

\begin{corollary}
    Let $T$ be a Latin square of order $n$, where $n \geq 3$. Then $\mathscr{G}(T)$ is regular of degree $3\binom{n}{2}.$
\end{corollary}

\begin{proof}
    We note that $T$ is a Latin tableau which has $n$ rows, all of the same length, and $n$ columns, also all of the same length. Hence in the statement of Theorem \ref{count_pairs}, we have $a = b = \binom{n}{2}$. Since all columns have length $\geq 3,$ the tableau $T$ cannot have a symmetric pair of columns, so $p = 0$ and the proof is complete.
\end{proof}

We remind the reader that for a Young diagram $\lambda$, we denote with $\lambda_1$ the length of the first row (which is also the biggest row). In what follows, we denote by $\Lambda_n$ the set of all Young diagrams such that $\lambda_1 = n$. Observe that in particular the Young diagram $\lambda$ corresponding to a Latin square of size $n \geq 3$ is in $\Lambda_n$ and each vertex of $\mathscr{G}(\lm) $ has degree $3\binom{n}{2}$. For any other Young diagram  $\lambda \in \Lambda_n$, the maximum vertex degree of $\mathscr{G}(\lambda)$ is less than or equal to $3\binom{n}{2}$. We also notice that the staircase shape $\lambda = (n, n-1, ..., 1)$ has an isotopy graph consisting of a single vertex, with degree zero. It is an open problem to find out for which numbers $0 \leq \ell \leq 3\binom{n}{2}$ does there exist a wide partition $\lambda$ in $\Lambda_n$ such that the maximum vertex degree of $\mathscr{G}(\lambda)$ is $\ell$. For all $n$, the partition consisting of a single row of length $n$ has an isotopy graph which is regular of degree $\binom{n}{2}$. When $n$ is even and $n \geq 4$, the tableau $T_n$ constructed in the next section provides an example of a tableau in $\Lambda_n$ whose isotopy graph has maximum degree $n$.

\section{Cubes in isotopy graphs}
\label{section:cube}

\begin{definition} We say a graph $G$ is \emph{isomorphic to a d-cube} (or simply that $G$ is a $d$-cube) if there is a bijection between vertices of $G$ and $d$-tuples with entries in $\{0,1\}$ such that two vertices are adjacent if and only if the corresponding $d$-tuples differ in a single coordinate. 
\end{definition}

\begin{figure}[ht]
    \centering
    \includegraphics[trim=3in 2.75in 3in 1.75in, clip]{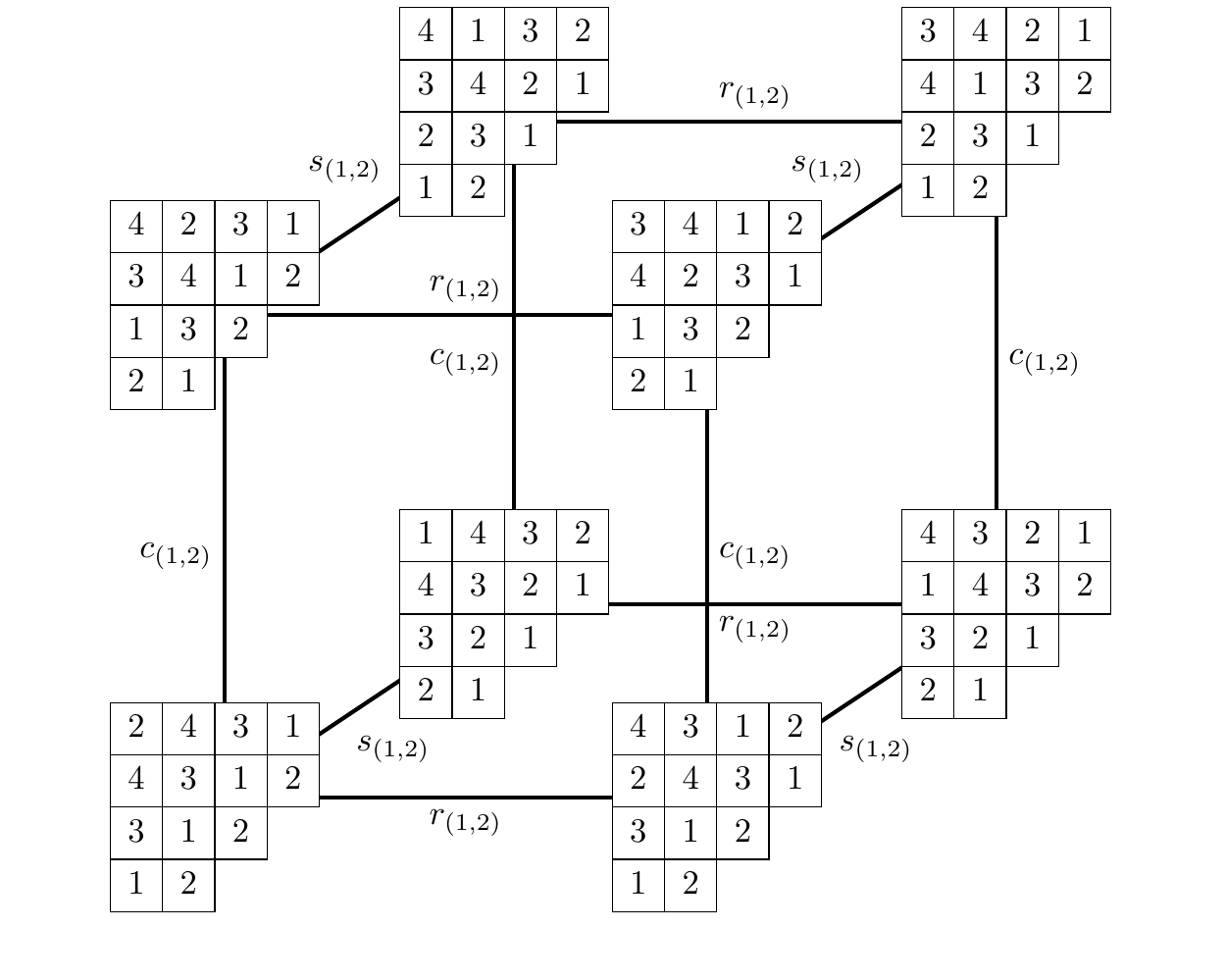}
    \caption{A graph isomorphic to a $3$-cube.}
    \label{fig:cube}
\end{figure}

The goal of this section is to characterize tableaux whose isotopy graphs are isomorphic to a cube, and construct isotopy graphs isomorphic to cubes of all dimensions. We will see that if a tableau $T$ of shape $\lm$ has an isotopy graph which is a cube, then $\lm$ must satisfy the following condition.

\begin{definition}
    If a shape $\lm$ has no three rows and no three columns of the same length, we call $\lm$ a \textbf{squareable} Young diagram. A Latin tableaux is squareable if its shape is a squareable Young diagram.
\end{definition}

The next lemma will be used in the proof of Lemma \ref{lem:squareable} below, which states that a tableau whose isotopy graph is a cube must be squareable.

\begin{lemma} \label{commuting}
    All elementary transformations acting on a Latin tableau $T$ of shape $\lm$  pairwise commute if and only if $\lm$ is squareable.
\end{lemma}

\begin{proof}
    If $r \in S_{\text{row}}$, $c \in S_{\text{col}}$ and $s \in S_{\text{ent}}$ are elementary transformations, then clearly $r, c$ and $s$ commute with each other.  Hence it suffices to show that the elementary transformations in $S_{\text{row}}$ pairwise commute if and only if $\lm$ has no three rows of the same length; and that the elementary transformations in $S_{\text{col}}$ and $S_{\text{ent}}$ respectively pairwise commute if and only if $\lm$ has no three columns of the same length. 
    
     Suppose $\lm$ has three rows $a,b$ and $c$ of the same length.  Then $r_{(a,b)}$ and $r_{(b,c)}$ do not commute, proving the forward direction of the statement for $S_{\text{row}}.$ Conversely, suppose $r_{(i,j)}, r_{(x,y)} \in S_{\text{row}}$ are distinct elementary transformations which do not commute. Then $\{i,j\} \cap \{x,y\}$ must be nonempty. By definition of $S_{\text{row}}$, rows $i$ and $j$ have the same length; and so do rows $x$ and $y$.  Since $\{i,j\} \cap \{x,y\}$ is nonempty, rows $i,j$ and $x$ all have the same length. This proves the reverse direction of the statement about $S_{\text{row}}$.
   
    The argument for $S_{\text{col}}$ is similar. Finally, by an analogous argument, we may show that elementary transformations in $S_{\text{ent}}$ pairwise commute if and only if no three entries of $T$ appear the same number of times. To finish the proof, note that a Latin tableau $T$ of shape $\lm$ has three entries which appear the same number of times if and only if the shape $\lm$ has three columns of the same length. 
\end{proof}


    
    
    

\begin{lemma}
    \label{lem:squareable}
    Let $T$ be a Latin tableau.  If $\mathscr{G}(T)$ is isomorphic to a cube, then $\lm$ is squareable.
\end{lemma}

\begin{proof}
    Suppose $T$ is not squareable, and suppose $T$ has three columns of the same length: columns $x,y\text{ and }z$.  (If $T$ has three rows of the same length, the argument is similar.)  Let $C$ be the subgroup of the isotopy group of $T$ which acts by permuting columns $x,y \text{ and }z.$ Then $C$ is generated by the elementary transformations $c_{(x,y)}$  $c_{(x,z)}$ and $c_{(y,z)},$ and $C$ acts faithfully on $\mathscr{G}(T)$.  Hence the Cayley graph $G$ of $C$ with those generators appears as a subgraph of $\mathscr{G}(T)$.  We will prove that there is a pair of vertices in $G$ which have at least three common neighbors.  To see this, note that 
    \[c_{(y,z)}c_{(x,y)} = c_{(x,z)}c_{(y,z)} = c_{(xzy)}\]
    \[c_{(x,y)}c_{(y,z)} = c_{(x,z)}c_{(x,y)} = c_{(xyz)}\]
    Hence in the Cayley graph $G$, $c_{(x,y)}$ and $c_{(y,z)}$ have three distinct common neighbors: the identity, $c_{(y,z)}c_{(x,y)}$, and $c_{(x,y)}c_{(y,z)}$.
    However, it easy to check that no pair of vertices in a cube has more than two common neighbors. This shows that $\mathscr{G}(T)$ is not a cube, and the proof is complete.
 \end{proof}

\begin{remark}
    It is tempting to conjecture that the isotopy graph of any tableau whose shape is squareable must be a cube.  However, this is not the case, as the following example shows.
\end{remark}

Let $\lm$ be the Young diagram of shape 
    \[(2k, 2k, 2(k-1), 2(k-1), \ldots, 4,4,2,2).\]
    Then $\lm$ can be partitioned into $2 \times 2$ squares, defined by intersecting a pair of rows of the same length with a pair of columns of the same length. In a Latin tableau of shape $\lm$ each of these $2 \times 2$ squares must have entries 
    \[\young(xy,yx).\]
    for some $x,y$
    For example, when $k=3$, one tableau of shape $\lm$ is shown in Figure \ref{fig:symmetries}.
    \begin{figure}
    \[\young(654321,563412,4321,3412,21,12).\]
    \caption{A tableau with nontrivial symmetries}
    \label{fig:symmetries}
    \end{figure}
    Let $T$ be a Latin tableau of shape 
     \[\lm = (2k, 2k, 2(k-1), 2(k-1), \ldots, 4,4,2,2)\]
     for some $k$.  Then each vertex in the isotopy graph of $T$ has degree $3k$. So if $\mathscr{G}(T) \cong \mathscr{C}^d$ for some $d$, we must have $d = 3k$.
     
    Suppose we now switch every pair of rows of the same length, and then switch every pair of columns of the same length.  The overall effect is to switch the boxes in opposite corners of each $2 \times 2$ square, which does not change $T$ by the above discussion.  Hence $T$ has a nontrivial symmetry.
    
     The isotopy group of $\lm$ has order $2^{3k}$.  Since $T$ has at least one nontrivial symmetry, $\mathscr{G}(T)$ has fewer than $2^{3k}$ vertices.  So $\mathscr{G}(T)$ is not isomorphic to a cube, and no Latin tableau of shape $\lm$ has an isotopy graph isomorphic to a cube.

\begin{definition}
     Suppose columns $i$ and $j$ form a symmetric pair in a tableau $T$, and the columns have entries in $\{x,y\}$.  Then $T$ has a nontrivial symmetry $\rho = s_{(x,y)}c_{(i,j)}$. We say a symmetry $\sigma$ of a latin tableau $T$ \emph{arises from symmetric pairs of columns} if $\sigma$ is a composition of symmetries of the form given above.
\end{definition}

\begin{lemma}
\label{cayley}
Suppose $T$ is a squareable Latin tableau such that the following hold:
\begin{enumerate}
    \item The degree of a vertex in $\mathscr{G}(T)$ is $d$.
    \item The Latin tableau $T$ has no nontrivial symmetries, except perhaps some arising from symmetric pairs of columns.
\end{enumerate}
Then $\mathscr{G}(T) \cong \mathscr{C}^d$.
\end{lemma}

\begin{proof}
    Let $\lm$ be the shape of $T$. Let $t_1,\ldots,t_d$ be elementary transformations which generate the isotopy group $\mathscr{S}(\lm)$.  Let $\mathbb{Z}_2$ be the additive group of integers modulo $2$.  Since $\lm$ is squareable, we have an isomorphism 
    \[\varphi: (\mathbb{Z}_2) ^ d \rightarrow \mathscr{S}(\lm)\]
    defined by 
    \[(a_1,a_2,\ldots,a_d) \mapsto t_1^{a_1}t_2^{a_2}\cdots t_d^{a_d}\]
    where $a_i \in \{0, 1\}$.  
    
    Suppose $T$ has no nontrivial symmetries, and consider the Cayley graph of $(\mathbb{Z}_2) ^ d$ with respect to the generating set $\{e_1,\ldots,e_d\}$
    where $e_i$ is the $d$-tuple with a $1$ in the $i^{th}$ coordinate, and $0$'s everywhere else.  
    The vertices of this graph are in obvious bijection with vertices of $\mathscr{C}^d$, and two vertices are adjacent if and only if they differ in a single coordinate.  Hence the Cayley graph of $(\mathbb{Z}_2)^d$ is a $d$-dimensional cube. 
    
    The isomorphism $\varphi$ carries the generator $e_i$ of $(\mathbb{Z}_2) ^ d$ to the elementary transformation $t_i$ in $\mathscr{S}(\lm)$.  Hence this isormophism induces an isomorphism of Cayley graphs from the cube $\mathscr{C}^d$ to the Cayley graph of $\mathscr{S}(\lm)$ with generating set $\{t_1,\ldots,t_d\}$.  But since $T$ has no nontrivial symmetries, $\mathscr{G}(T)$ is just the Cayley graph of $\mathscr{S}(\lm)$ with respect to the generating set $t_1,\ldots,t_n$.  This completes the proof.
    
    Next, suppose $T$ has at least one symmetric pair of columns. Note that since $\lm$ is squareable, $T$ has at most two symmetric pairs of columns, which must be disjoint: at most one symmetric pair of columns of length $2$, and at most one symmetric pair of columns of length $1$. 
    
    Suppose $T$ has exactly one symmetric pair of columns, consisting of columns $i$ and $j$ with entries in $\{x,y\}$. Then $T$ has a nontrivial symmetry $\rho = c_{(i,j)}s_{(x,y)}$ of order $2$, and the stabilizer $\stab(T)$ of $T$ in $\mathscr{S}(\lm)$ is the group of order 2 generated by $\rho$.
    
    For simplicity of notation, suppose the elementary transformations $c_{(i,j)}$ and $s_{(x,y)}$ correspond to $t_{d-1}$ and $t_d$ under the isomorphism $(\mathbb{Z}_2)^d \rightarrow \mathscr{S}(\lm)$ described above.
    
    Then $\mathscr{S}(T)/\stab(T)$ has a complete set of distinct coset representatives of the form 
    \[t_1^{a_1}t_2^{a_2}\cdots t_{d-1}^{a_{d-1}} \stab(T)\]
    as $a_1,\ldots,a_{d-1}$ range through $\{0,1\}$.
    
    For $i \in \{1,\ldots,d-1\}$, the elementary transformation $t_i$ acts on 
     \[t_1^{a_1}t_2^{a_2}\cdots t_{d-1}^{a_{d-1}} \stab(T)\]
     by changing $a_i$ either from $0$ to $1$ or from $1$ to $0$; while the elementary transformation $t_d$ acts by changing the value of $a_{d-1}$.  Hence two vertices of the coset graph are identical if and only if the tuples $(a_1,\ldots,a_{d-1})$  differ in exactly one coordinate, and the isotopy graph is a cube as desired. 
     
     If $T$ has instead two symmetric pairs of columns, the argument is similar. Here, suppose that 
     $\rho = c_{(i,j)}s_{(x,y)}$
      and $\rho' = c_{(k,\ell)}s_{(u,v)}$ are the symmetries corresponding to the two symmetric pairs of columns, and suppose for simplicity of notation that $c_{(i,j)} = t_{d-3},$ $s_{(x,y)} = t_{d-1},$  $c_{(k,\ell)} = t_{d-2}$ and $s_{(u,v)} = t_d.$
     Then the argument is analogous to the case with a single symmetric pair, except that our set of coset representatives is given by
      \[t_1^{a_1}t_2^{a_2}\cdots t_{d-2}^{a_{d-2}} \stab(T).\]
      The transformations $t_1,\ldots,t_{d-2}$ act by toggling the exponents $a_i$ between $0$ and $1$, while $t_{d-1}$ and $t_{d}$ change the value of $a_{d-3}$ and $a_{d-2}$ respectively. Once again, the isotopy graph is a cube, and the proof is complete in this case.
\end{proof}

We know from Figure \ref{fig:symmetries} that not every squareable Young diagram has a filling that has a cubical isotopy graph. In the next Theorem we complete the characterization of when squareable Latin tableaux have an isotopy graph being a cube. 

\begin{theorem}\label{cube_criterion}
    For $T$ a Latin tableau of shape $\lm$, the graph $\mathscr{G}(T)$ is a cube if and only if both of the following hold.
    \begin{enumerate}
        \item The shape $\lm$ is squareable.
        \item The tableau $T$ has no nontrivial symmetries, except those arising from symmetric pairs of columns.
    \end{enumerate}
\end{theorem}

\begin{proof}
    From Lemma \ref{cayley}, we know that if the above conditions hold, then $\mathscr{G}(T)$ is a cube.  From Lemma \ref{lem:squareable}, we know that the first condition is necessary for $\mathscr{G}(T)$ to be a cube. It remains to prove that if $T$ has a nontrivial symmetry which does not arise from symmetric pairs of columns, then $\mathscr{G}(T)$ is not a cube. Let $\lm$ be a squareable Young diagram, and let $d$ be the number of elementary transformations acting on $T$.
    Since $\lm$ is squareable, the isotopy group of $\lm$ is isomorphic to $(\mathbb{Z}_2)^d$, with each factor generated by a distinct elementary transformation.
    
    Let $d'$ be the degree of a vertex in $\mathscr{G}(T)$. 
    Then by Lemma \ref{count_pairs}, $d' = d - p$, where $p$ is the number of symmetric pairs of columns in $T$. 
    Note that since $\lm$ is squareable, $T$ has at most two symmetric pairs of columns, which must be disjoint: at most one symmetric pair of columns of length $2$, and at most one symmetric pair of columns of length $1$. 
    
    Suppose columns $i$ and $j$ form a symmetric pair in $T$, with entries in $\{x,y\}$.  Then $T$ has a nontrivial symmetry $\rho = c_{(i,j)}s_{(x,y)}$ of order $2$.  If $T$ has a second symmetric pair, then that pair contributes a second nontrivial symmetry $\rho'$, and $\rho$ and $\rho'$ generate a subgroup of $\mathscr{S}(T)$ which has order $4.$   Hence if $p$ is the number of symmetric pairs of columns in $d$, we have $|\stab(T)| \geq 2^p$, with equality holding if and only if all the symmetries of $T$ arise from symmetric pairs of columns.
    
    The number of vertices in $\mathscr{G}(T)$ is given by
    \[|\text{vert}(\mathscr{G}(\lm))| = \frac{|\mathscr{S}(\lm)|}{|\stab(T)|} = \frac{2^d}{|\stab(T)|}.\]
    Hence by the discussion in the previous paragraph,
    \[|\text{vert}(\mathscr{G}(\lm)| \leq \frac{2^d}{2^p} = 2^{d-p} = 2^{d'}\]
    with equality if and only if the only symmetries of $T$ are those arising from symmetric pairs of columns. 
   Now every vertex of $\mathscr{G}(T)$ has degree $d'$.  Hence if $\mathscr{G}(T)$ is a cube, then $\mathscr{G}(T)$ is in particular a $d'$-dimensional cube, so we must have $|\text{vert}(\mathscr{G}(\lm))| = 2^{d'}$.  Hence if $T$ has nontrivial symmetries which do not arise from symmetric pairs of columns, $\mathscr{G}(T)$ cannot be a cube, and the proof is complete.
\end{proof}

\begin{theorem}\label{theorem:cube}
    For every positive integer $d$, there exists a Latin tableau $T$ such that $\mathscr{G}(T) \cong \mathscr{C}^d$.
\end{theorem}

\begin{proof}

     For $d \leq 3$, we give specific tableaux with the desired isotopy graphs.  We then give a general construction for all larger values of $d$.

    For $d=0$, we take the tableau 
    \[T_0=\young(1).\]
    The isotopy graph of $T_0$ consists of a single vertex, as desired. 
    
    For $d=1$, we take 
    \[T_1=\young(12).\]
    Here $s_{(1,2)}$ and $c_{(1,2)}$ act identically on $T$, and $\mathscr{G}(T)$ has a single edge, labeled with both $s_{(1,2)}$ and $c_{(1,2)}$. So $\mathscr{G}(T) \cong \mathscr{C}^1$.
    
    For $d=2$, we take 
    \[T_2=\young(312,12).\]
    The elementary transformations acting on $T_2$ are $c_{(1,2)}$ and $s_{(1,2)}$, and it is easy to check that composing them does not give a symmetry of $T_2$.
    
     If $d = 3$, we can check that the tableau \[T_3=\young(2431,4312,312,12)\]
    has the desired isotopy graph.  The elementary transformations are $r_{(1,2)}$, $c_{(1,2)}$ and $s_{(1,2)}$.
    
    We now give a construction for all higher values of $d$.
    By the previous lemma, it is enough to find a squareable Young diagram with $d$ elementary transformations in its isotopy group, and no nontrivial symmetries.
    
   \begin{figure}[h]
        \centering
        \begin{subfigure}[b]{0.4\textwidth}
        \young(75312468,531246,3124,12)
        \caption{The tableau $T_8$}
        \label{fig:lamda_8}
        \end{subfigure}
        \begin{subfigure}[b]{0.4\textwidth}
        \young(87531246,75312468,531246,3124,12)
        \caption{The tableau $T_9$}
        \label{fig:lambda_9}
        \end{subfigure}
        \caption{Tableaux with isotopy graphs isomorphic to cubes.}
    \end{figure}
   We first address the case where $d$ is even.  For even $d \geq 4$, let $\lm_d$ be the Young diagram with $\frac{d}{2}$ rows and $d$ columns, such that each row of $\lm_d$ has two fewer boxes than the one above it. Then $T_d$ is the filling of $\lm_d$ such that the $k^{th}$ row from the \emph{bottom} has entries
    \[2k-1, 2k-3, \cdots,3, 1, 2, 4, \ldots, 2k-2, 2k.\]
    (Here row $\frac{d}{2}$ is the \emph{first} row from the bottom, and so has entries $1,2$.)
    
   Note that $\lm_d$ is squareable, with no two rows of the same length and $\frac{d}{2}$ pairs of columns of the same length.  Hence there are $d$ elementary transformations that generate the isotopy group of $\lm_d$, of which $\frac{d}{2}$ act on columns and $\frac{d}{2}$ act on entries.
    
    Let $d \geq 4$ be even, and let $\rho$ be an element of $\mathscr{S}(T_d)$ which acts trivially on the tableau $T_d$.
    We can write $\rho$ in the form $\rho = \rho_{\text{ent}}\rho_{\text{col}}$ where $\rho_{\text{ent}}$ acts on $T_d$ by permuting entries, and $\rho_{\text{col}}$ acts by permuting columns.
    
    Suppose $\rho_{\text{col}}$ switches columns $j$ and $j+1$ in $T_d$.  Then $\rho_{ent}$ must permute the entries of columns $j$ and $j+1$, so as to restore each entry to its original position.  This is only possible if, for each row $i$ which intersects column $j$, the boxes $T_d$ in positions $(i,j)$ and $(i,j+1)$ are consecutive integers.  This does not hold for any column of $T_d$, by construction.  
    
    Hence $\rho_{\text{col}}$ must be the identity, and $\rho$ acts on $T_d$ strictly by permuting entries.  Since $\rho$ acts trivially on $T_d$, it follows that $\rho$ is the identity, and $T_d$ has no trivial symmetries, as desired.
    
    Now, let $d$ be odd, and suppose $d > 3$.  Then $d-1$ is an even number which is greater than or equal to $4$.  
    Let $T_d$ by the tableau whose first row has entries
    \[d-1, d-2, d-4, \cdots, 3, 1, 2, 4, \cdots d-5, d-3\]
    and whose remaining rows form a copy of the tableau
    $T_{d-1}$ defined above for the even case.  
    
    Let $\lm_d$ be the shape of $T_d$.  Note that $\lm_{d}$ has $(d-1)/2$ pairs of columns of the same length, and a single pair of rows of the same length.  This gives $d$ distinct elementary transformations acting on $\lm_d$.  We must show that the tableau $T_d$ has no nontrivial symmetries.
    
    Let $\rho \in \mathscr{S}(\lm_d)$, and suppose $\rho$ acts trivially on $T_d$.  As in the even case, we can write $\rho$ as a product $\rho = \rho_{\text{col}}\rho_{\text{ent}}\rho_{\text{row}}$, where $\rho_{\text{row}}$ is either $r_{(1,2)}$ or the identity.
    
    Note that $\rho_{\text{col}}$ permutes the first $d-3$ columns of $T_d$ amongst themselves, and similarly $\rho_{\text{ent}}$ permutes entries in the set $\{1,2,\ldots,d-3\}$ amongst themselves.  Hence $\rho$ acts by an isotopy on the copy of $T_{d-3}$ which is obtained by removing the first two rows of $T_d$.  Since $T_{d-3}$ has no nontrivial symmetries, it follows that $\rho_{\text{col}}$ is either the identity or the elementary transformation $c_{(d-2,d-1)}$.  Similarly $\rho_{\text{ent}}$ is either the identity or the elementary transformation $s_{(d-1,d-2)}$.
    
    So $\rho = c_{(d-2,d-1)}^{a_1}
                s_{(d-1,d-2)}^{a_2}
                r_{(1,2)}^{a_3}$ where 
    $a_1,a_2,a_3 \in \{0,1\}$.
    If $a_3 = 1$, then $\rho(T_d)$ has the same entry in positions $(1,1)$ and $(2,2)$.  This is a contradiction, since $\rho$ acts trivially on $T_d$.  If $a_2 = 1$, then $\rho(T_d)$ contains a total of two entries equal to $d-1$ in its first two columns, compared to one such entry for $T_d,$ again a contradiction. So $a_2 = a_3 = 0$.  Since a single elementary transformation cannot act trivially, the only possibility is that $\rho$ is the identity.  Hence $T_d$ has no nontrivial symmetries for odd values of $d$, and the proof is complete.
\end{proof}

\appendix

\section{Young diagrams with triangles in their isotopy graphs}\label{Appendix}

In this appendix, we list all Young diagrams whose isotopy graphs contain triangles. For each shape $\lm$, we give an example of a filling $T$ of $\lm$ such that $\mathscr{S}(T)$ contains a triangle.  For the first two examples, the remaining two vertices of the triangle are obtained by switching the last two rows and last two columns of the tableau, respectively.  For the remaining examples, we highlight the four key boxes described in Proposition \ref{prop:triangle}. When ellipses appear in the diagram, the indicated rows may be extended as much as we like to give a Young diagram, and the additional boxes filled with entries in an obvious way to give a Latin tableau, without impacting the presence of a triangle in the isotopy graph. 

By Theorem \ref{clique_number}, each triangle in an isotopy graph is part of a maximal 4-clique. We can see how such a four-clique arises in each example. Let $r$ be the elementary transformation that switches the two rows containing highlighted boxes, and let $c$ be the transformation that switches the two columns containing these boxes. In each case, applying the transformation $rc$ to the tableau is equivalent to applying an elementary transformation $s$ that switches two entries.  We have a triangle with vertices $T$, $rT$ and $cT$ by Proposition \ref{prop:triangle}. The tableaux $sT$ gives the fourth vertex of a 4-clique. All remaining edges can be easily filled in, using the fact that $r,c$ and $s$ commute.

In some cases, the given tableaux are contained in additional triangles in the isotopy graph, which do not arise for the four highlighted boxes.  Each of these triangles is part of a $4$-clique, which can be found in the same way as above.

As a final note, we have given only one possible filling for each Young diagram.  There may be other fillings of the same shape that produce structurally different isotopy graphs, which also contain triangles. 

\vspace{0.25in}

\[\ytableaushort{3412{\none[\ldots]}, 4321, 12, 21}
\qquad 
\ytableaushort{534{\none[\ldots]}12{\none[\ldots]}, 34{\none[\ldots]}{\none[\ldots]}21, 12, 21}
\]

\hspace{0.25in}

\[\ytableaushort{231, 312, 12}
    *[*(lightgray)]{2,2}
    \qquad
    \ytableaushort{3412,4321,123,21}
    *[*(lightgray)]{2,2}
    \qquad
    \ytableaushort{213, 321}
    * [*(lightgray)]{1+2,1+2}
    \qquad
     \ytableaushort{213, 321, 1}
    * [*(lightgray)]{1+2,1+2}
    \qquad
    \ytableaushort{2143, 3412, 123}
    *[*(lightgray)]{1+2,1+2}
\]

\hspace{0.25in}

\[\ytableaushort{3412, 4321}
    * [*(lightgray)]{2+2,2+2}
    \qquad
    \ytableaushort{3412, 4321, 1}
    *[*(lightgray)]{2+2,2+2}
    \qquad 
    \ytableaushort{3412, 4321, 12}
    *[*(lightgray)]{2+2,2+2}
    \qquad
    \ytableaushort{3412,4321,21,1}
    *[*(lightgray)]{2+2,2+2}
\]

\hspace{0.25in}

\[\ytableaushort{3412{\none[\ldots]}, 12, 21}
    *[*(lightgray)]{0,2,2}
    \qquad
    \ytableaushort{132{\none[\ldots]},213,321}
    *[*(lightgray)]{0,1+2,1+2}
    \qquad
    \ytableaushort{4321{\none[\ldots]},213,321,1}
    *[*(lightgray)]{0,1+2,1+2}
    \qquad
    \ytableaushort{4213{\none[\ldots]},2341,1432,312}
    *[*(lightgray)]{0,1+2,1+2}
\]

\hspace{0.25in}

\[\ytableaushort{1234{\none[\ldots]}, 3412, 4321}
    *[*(lightgray)]{0,2+2,2+2}
    \qquad
    \ytableaushort{1234{\none[\ldots]}, 3412, 4321, 21}
    *[*(lightgray)]{0, 2+2, 2+2}
    \qquad
    \ytableaushort{52341{\none[\ldots]}, 3412, 4321, 21, 1}
    *[*(lightgray)]{0, 2+2, 2+2}
    \qquad
    \ytableaushort{563412{\none[\ldots]}, 3412, 4321, 21, 12}
    *[*(lightgray)]{0, 2+2, 2+2}
\]

\hspace{0.25in}

\[\ytableaushort{563412{\none[\ldots]}, 3412, 4321, 21, 12}
  *[*(lightgray)]{0, 2+2, 2+2}
  \qquad\ytableaushort{34521{\none[\ldots]},4321   {\none[\ldots]} ,12,21}
  *[*(lightgray)]{0, 0, 2, 2}
  \qquad
  \ytableaushort{1234,2143,3412,4321}
    *[*(lightgray)]{0,0,2+2,2+2}
  \qquad
  \ytableaushort{52341{\none[\ldots]},2143{\none[\ldots]},3412     ,4321}
    *[*(lightgray)]{0,0,2+2,2+2}
\]

\hspace{0.25in}

\[\ytableaushort{52341{\none[\ldots]},2143{\none[\ldots]},3412,4321, 1}
    *[*(lightgray)]{0,0,2+2,2+2}
    \qquad
   \ytableaushort{25341, 51432, 3412, 4321, 12}
   *[*(lightgray)]{0,0,2+2, 2+2}
   \qquad
   \ytableaushort{563412{\none[\ldots]}, 2143{\none[\ldots]}, 3412, 4321, 12}
   *[*(lightgray)]{0,0,2+2, 2+2}
\]

\hspace{0.5in}
    
\[\ytableaushort{653421{\none[\ldots]},52431{\none[\ldots]},3412,4321,21,1}
    *[*(lightgray)]{0,0,2+2,2+2}
    \qquad
    \ytableaushort{653421,564312,3412,4321,21,12}
    *[*(lightgray)]{0,0,2+2,2+2}
    \qquad
    \ytableaushort{6534721{\none[\ldots]},564321{\none[\ldots]},3412,4321,21,12}
    *[*(lightgray)]{0,0,2+2,2+2}
\]

\bibliographystyle{plain}

\bibliography{bib}

\begin{thebibliography}{10}

\bibitem{AT92}
N.~Alon and M.~Tarsi.
\newblock Coloring and orientations of graphs.
\newblock {\em Combinatorica}, 12:125--143, 1992.

\bibitem{CFGV02}
Timothy Chow, C.~Fan, Michel Goemans, and Jan Vondr\'{a}k.
\newblock Wide partitions, {L}atin tableaux, and {R}ota's basis conjecture.
\newblock {\em Advances in Applied Mathematics}, 31, 06 2002.

\bibitem{Dri97}
Arthur~A Drisko.
\newblock On the number of even and odd latin squares of order $p+1$.
\newblock {\em Advances in Mathematics}, 128(1):20 -- 35, 1997.

\bibitem{Dri98}
Arthur~A. Drisko.
\newblock Proof of the {A}lon-{T}arsi {C}onjecture for $n = 2^r p$.
\newblock {\em The Electronic Journal of Combinatorics}, 5, 1998.

\bibitem{DG13}
Christoph D{\"u}rr and Flavio Gu{\'i}{\~n}ez.
\newblock The wide partition conjecture and the atom problem in discrete
  tomography.
\newblock {\em Electronic Notes in Discrete Mathematics}, 44:351 -- 356, 2013.

\bibitem{falcon2019enumerating}
Ra{\'u}l~M. Falc{\'o}n and Rebecca~J. Stones.
\newblock Enumerating partial latin rectangles.
\newblock {\em arXiv preprint arXiv:1908.10610}, 2019.

\bibitem{Gly10}
David~G. Glynn.
\newblock The conjectures of {A}lon-{T}arsi and {R}ota in dimension prime minus
  one.
\newblock {\em SIAM Journal on Discrete Mathematics}, 24:394--399, 2010.

\bibitem{Hua94}
Rosa Huang and Gian-Carlo Rota.
\newblock On the relations of various conjectures on {L}atin squares and
  straightening coefficients.
\newblock {\em Discrete Mathematics}, 128(1):225 -- 236, 1994.

\bibitem{HKO11}
Alexander Hulpke, Petteri Kaski, and Patric {\"O}sterg{\aa}rd.
\newblock The number of latin squares of order 11.
\newblock {\em Mathematics of computation}, 80(274):1197--1219, 2011.

\bibitem{MMM07}
Brendan~D. McKay, Alison Meynert, and Wendy Myrvold.
\newblock Small {L}atin squares, quasigroups, and loops.
\newblock {\em Journal of Combinatorial Designs}, 15(2):98--119, 2007.

\bibitem{MW05}
Brendan~D. McKay and I.~M. Wanless.
\newblock On the number of {L}atin squares.
\newblock {\em Annals of Combinatorics}, 9:335--344, 2005.

\bibitem{San14}
Sangwook Ree.
\newblock Confucian cholar's discovery predates the work of {E}uler.
\newblock In {\em Math\&Presso No.3}, International Conference of
  Mathematicians, August 2014.

\bibitem{Sto13}
Douglas~S. Stones.
\newblock Symmetries of partial {L}atin squares.
\newblock {\em European Journal of Combinatorics}, 34(7):1092 -- 1107, 2013.

\bibitem{SW12}
Douglas~S. Stones and Ian~M. Wanless.
\newblock How not to prove the {A}lon-{T}arsi conjecture.
\newblock {\em Nagoya Mathematical Journal}, 205:1--24, 2012.

\end{thebibliography}

\end{document}